\documentclass[12pt,twoside]{amsart}

\usepackage{amssymb}
\usepackage{verbatim}

\numberwithin{equation}{section}
\newtheorem{theorem}{Theorem}[subsection]
\newtheorem{lemma}[theorem]{Lemma}
\newtheorem{proposition}[theorem]{Proposition}
\newtheorem{corollary}[theorem]{Corollary}

\theoremstyle{definition}

\newtheorem{remark}[theorem]{Remark}
\newtheorem{example}[theorem]{Example}

\newcommand{\popo}{\mathbb{P}^1 \times \mathbb{P}^1}
\newcommand{\pnr}{\mathbb{P}^{n_1}\times \cdots \times \mathbb{P}^{n_t}}

\newcommand{\N}{\mathbb{N}}

\newcommand{\pr}{\mathbb{P}}

\newcommand{\NEF}{\operatorname{NEF}}

\newcommand{\EFF}{\operatorname{EFF}}

\begin{document}


\title[Symbolic powers and regular powers]{
Symbolic powers versus regular powers of ideals
of general points in $\popo$}

\author{Elena Guardo}
\address{Dipartimento di Matematica e Informatica\\
Viale A. Doria, 6 - 95100 - Catania, Italy}
\email{guardo@dmi.unict.it}
\urladdr{http://www.dmi.unict.it/$\sim$guardo/}

\author{Brian Harbourne}
\address{Department of Mathematics \\
University of Nebraska--Lincoln\\
Lincoln, NE 68588-0130, USA }
\email{bharbour@math.unl.edu}
\urladdr{http://www.math.unl.edu/$\sim$bharbourne1/}

\author{Adam Van Tuyl}
\address{Department of Mathematical Sciences \\
Lakehead University \\
Thunder Bay, ON P7B 5E1, Canada}
\email{avantuyl@lakeheadu.ca}
\urladdr{http://flash.lakeheadu.ca/$\sim$avantuyl/}

\thanks{Version: February 21, 2012}

\begin{abstract}
Recent work of Ein-Lazarsfeld-Smith and Hochster-Huneke
raised the problem of which symbolic powers of an ideal
are contained in a given ordinary power of the ideal.
Bocci-Harbourne developed methods to address this problem,
which involve asymptotic numerical characters of
symbolic powers of the ideals. Most of the work
done up to now has been done for ideals defining 0-dimensional
subschemes of projective space.
Here we focus on certain subschemes given by
a union of lines in $\pr^3$ which can also be viewed
as points in $\popo$.
We also obtain results on the
closely related problem, studied by Hochster and by Li-Swanson, of
determining situations for which
each symbolic power of an ideal is an ordinary power.
\end{abstract}

\maketitle


\renewcommand{\thetheorem}{\thesection.\arabic{theorem}}
\setcounter{theorem}{0}

\section{Introduction}

Refinements of the groundbreaking results of \cite{refELS, HaHu}
regarding which symbolic powers of ideals are contained in a given
ordinary power of the ideal have recently been given in \cite{BH,
BH2,BCH, HH}, with a focus on ideals defining 0-dimensional
subschemes of projective space. The methods mainly involve giving
numerical criteria, both for containment and for non-containment.
These criteria have been extended in \cite{GHVT} to ideals
defining smooth subschemes in $\pr^N$ and applied to the case of
disjoint unions of lines. The most difficult numerical character
needed for these results is denoted in these papers by
$\gamma(I)$. We pause briefly to recall its definition.

Throughout this paper we work over an algebraically closed field
$k$ of arbitrary characteristic. Let $k[\pr^N]$ denote the
polynomial ring $k[x_0,\ldots,x_N]$ with the standard grading (so
each variable has degree 1). Given any homogeneous ideal $(0)\neq
I\subseteq k[\pr^N]$, $\alpha(I)$ denotes the least degree of a
nonzero form (i.e., homogeneous element) in $I$. Then the limit
$\lim_{m\to\infty}\alpha(I^{(m)})/m$ is known to exist (see, for
example, \cite[Lemma 2.3.1]{BH}), and is denoted by $\gamma(I)$.

A large amount of work has been done studying $\gamma(I)$, in a
range of contexts (including number theory \cite{ refCh, refW,
refW2}, complex analysis \cite{refSk}, algebraic geometry
\cite{BH, BH2, refEV,  refSch} and commutative algebra
\cite{HaHu}), with an emphasis on the case that $I$ defines a
0-dimensional subscheme. Our focus here will be on computing
$\gamma(I)$ for ideals of lines in $\pr^3$. A special case for
which $\gamma(I)$ can be computed is when the symbolic powers
$I^{(m)}$ and ordinary powers $I^m$ all coincide. This is because
if $I^{(m)}=I^m$ for all $m\geq 1$, then
$\alpha(I^{(m)})=\alpha(I^m)=m\alpha(I)$, hence
$\gamma(I)=\alpha(I)$. Thus we will also be interested in
distinguishing when $I^{(m)}=I^m$ for $m\geq 1$ occurs and when it
doesn't. It has been known for a long time that $I^{(m)}=I^m$
holds for all $m\geq 1$ when $I$ is a complete intersection (i.e.,
defined by a regular sequence; see \cite[Lemma 5, Appendix
6]{ZS}). What is of interest is when $I$ is not a complete
intersection. This also is a remarkably difficult problem; partial
results have been obtained for example by \cite{Ho, LS}.

The reason for our focus on ideals of certain unions of lines in $\pr^3$
is that for the cases we will consider the questions can be converted
into ones involving symbolic powers of ideals of points in $\popo$, using the fact
that we can regard a point in $\popo$ as being defined by a bigraded
ideal $I$ in $k[\popo]=k[x_0,\ldots,x_3]$. Since $k[\popo]=k[\pr^3]$ as rings,
we can regard $I$ as defining a subscheme of $\pr^3$, but the key is that
an ideal $I$ defining a point in $\popo$ when regarded as a bigraded ideal in
the bigraded ring $k[\popo]$, defines a line in $\pr^3$ when regarded as a
singly graded ideal in the usual grading on $k[\pr^3]$; see Remark \ref{ptsVlinesRem}.
Thus the ideal of a finite set of points in $\popo$ is simultaneously
(but with respect to a different grading) the ideal of a finite set of lines $\pr^3$.
(As a specific example, the ideal of  $s\leq4$ general points of $\popo$
is the ideal of $s$ general lines of $\pr^3$; see Remark \ref{ptsVlinesRem}.
For $s>4$, the ideal of $s$ points of $\popo$ is the ideal of $s$ lines
in $\pr^3$, but the lines are never general, even if the points are.)
Moving to $\popo$ makes available to us the vast array of work
done on products of projective spaces and surfaces in general, and on
$\popo$ in particular;  see, for example \cite{GuMaRa,Gu2,GVT,Ma,SVT,VTgen}.

Our main results are Corollary \ref{boundsCor} and Theorem \ref{thm1}.

\begin{corollary}\label{boundsCor}
Let $I$ be the ideal of $s\geq 1$ general points of $\popo$.
\begin{itemize}
\item[$\bullet$] If $s=1$, then $\gamma(I)=1$.
\item[$\bullet$] If $s=2$ or 3, then $\gamma(I)=2$.
\item[$\bullet$] If $s=4$, then $\gamma(I)=8/3$.
\item[$\bullet$] If $s=5$, then $\gamma(I)=3$.
\item[$\bullet$] If $s=6$, then $\gamma(I)=24/7$.
\item[$\bullet$] If $s=7$, then $\gamma(I)=56/15$.
\item[$\bullet$] If $s=8$, then $\gamma(I)=4$.
\item[$\bullet$] If $9\leq s$, then $\sqrt{s}-1<\gamma(I)\leq\sqrt{2s}$.
\end{itemize}
\end{corollary}

See section \ref{sect2} for the proof.

\begin{theorem}\label{thm1}
Let $I$ be the ideal of a set $Z$ of $s$ general points in $\popo$.
Then $I^m=I^{(m)}$ for all $m>0$ if and only if $s$ is $1, 2, 3$ or $5$.
Moreover, $I^{(4)}\neq I^3$ if $s=4$ and $I^{(2)}\neq I^2$ if $s\geq 6$.
\end{theorem}

See section \ref{sect3} for the proof. We note that the ideal
$I$ of $s$ general points of $\popo$ is a complete intersection if and only if $s=1$
(see the paragraph right before Proposition \ref{completeintersections}).

Whereas most of our focus in this paper is on sets of points in general position in $\popo$,
points not in general position can also be of interest; note for example that a reduced scheme
consisting of $s>1$ general points in $\popo$ is never arithmetically Cohen-Macaulay.
In a forthcoming paper we will study finite sets of points which are
arithmetically Cohen-Macaulay subschemes of $\popo$.

\noindent
{\bf Acknowledgments.}  We would like to thank Irena Swanson for answering
some of our questions.
This work was facilitated by the Shared Hierarchical
Academic Research Computing Network (SHARCNET:www.sharcnet.ca) and Compute/Calcul Canada.
Computer experiments carried out on {\tt CoCoA} \cite{C} and
{\it Macaulay2} \cite{Mt} were very helpful in guiding our research.
The second author's work on this project
was sponsored by the National Security Agency under Grant/Cooperative
agreement ``Advances on Fat Points and Symbolic Powers,'' Number H98230-11-1-0139.
The United States Government is authorized to reproduce and distribute reprints
notwithstanding any copyright notice.
The third author acknowledges the support provided by NSERC.


\renewcommand{\thetheorem}{\thesubsection.\arabic{theorem}}
\setcounter{theorem}{0}

\section{Background}\label{sect2}

\subsection{Points in $\popo$ and their ideals}
For the convenience of the reader, we begin with a review of
multi-graded ideals
arising in the context of products of projective space.

The multi-homogeneous coordinate ring $k[\pnr]$ of $\pnr$ is
$$k[x_{1,0},\ldots,x_{1,n_1},\ldots,x_{t,0},\ldots,x_{t,n_t}].$$
It has a multi-grading given by
$$\deg(x_{i,j})=e_i=(0,\ldots,0,1,0,\ldots,0) \in \N^t,$$
where the 1 is in the $i$th position.
The ring $k[\pnr]$ is a direct sum of its multi-homogeneous components
$k[\pnr]_{(a_1,\ldots,a_t)}$,
where $k[\pnr]_{(a_1,\ldots,a_t)}$ is the $k$-vector space span of the monomials
of multi-degree $(a_1,\ldots,a_t)$.
An ideal $I\subseteq k[\pnr]$ is multi-homogeneous if it is the direct sum
of its multi-homogeneous components (i.e., of $k[\pnr]_{(a_1,\ldots,a_t)}\cap I$).
Note that a multi-homogeneous ideal $I$ can be regarded as
a homogeneous ideal in $k[\pr^N]$, $N=n_1+\cdots+n_t+t-1$,
where a monomial of multi-degree $(a_1,\ldots,a_t)$ has degree
$d=a_1+\cdots+a_t$ and the homogeneous component of $I$ of degree $d$ is
$I_d=\bigoplus_{\sum_ia_i=d} I_{(a_1,\ldots,a_t)}$.
However, when  $t>1$, a multi-homogeneous ideal $I$
when regarded as being homogeneous
never defines a 0-dimensional subscheme of $\pr^N$,
even if $I$ defines a zero-dimensional subscheme of $\pnr$.
For example, the multi-homogeneous ideal $I$ of a finite set of points in $\popo$
defines a finite set of lines in $\pr^3$, which are skew (and thus not a cone)
if no two of the points lie on the same horizontal or vertical rule
of $\popo$ (see Remark \ref{ptsVlinesRem}), and not a complete intersection unless the points
comprise a rectangular array in $\popo$ (see the paragraph right before
Proposition \ref{completeintersections}).

Let $R = k[\popo]$, where we will use the standard multi-grading for $R$.
That is, $R = k[x_0,x_1,y_0,y_1]$,
with $\deg x_i = (1,0)$ and $\deg y_i = (0,1)$.
Let $I\subseteq R$ be a multi-homogeneous ideal (because
$R$ is bigraded, we sometimes say $I$ is bihomogeneous). Then $I$ has a
multi-homogeneous primary decomposition, i.e., a primary decomposition
$I=\bigcap_i Q_i$ where each $\sqrt{Q_i}$ is a multi-homogeneous
prime ideal, and $Q_i$ is multi-homogeneous and $\sqrt{Q_i}$-primary
\cite[Theorem 9, p. 153]{ZS}.
We define the $m$-th {\it symbolic power} of $I$ to be the ideal $I^{(m)}=\bigcap_j P_{i_j}$,
where $I^m=\bigcap_i P_i$ is a multi-homogeneous primary decomposition,
and the intersection $\bigcap_j P_{i_j}$ is over all components $P_i$ such that
$\sqrt{P_i}$ is contained in an associated prime of $I$.
In particular, we see that $I^{(1)}=I$ and that $I^m\subseteq I^{(m)}$.

Of particular interest to this paper is the case that $I$ is the ideal of a set
$Z$ of $s$ distinct
reduced points of $\popo$, i.e., $Z = \{P_1,\ldots,P_s\}$.   A point has the
form $P =[a_{0}:a_{1}] \times [b_{0}:b_{1}] \in \popo$ and its
defining ideal $I(P)$ in $R$ is a prime ideal of the form $I(P)
=(F,G)$ where $\deg F =(1,0)$ and $\deg G =(0,1)$.
The ideal $I(Z)$ is then given by $I(Z) = \bigcap_{i=1}^s I(P_i)$.  Furthermore,
the $m$-th symbolic power of $I(Z)$ has the form $I(Z)^{(m)} = \bigcap_{i=1}^s I(P_i)^m$.
The scheme defined by $I(Z)^{(m)}$ is sometimes referred to as a {\em fat point
scheme}, and denoted $mP_1+\cdots + mP_s$.

\begin{remark}\label{ptsVlinesRem}
Note that while the gradings on the rings $k[\popo]$ and $k[\pr^3]$
are different (and hence $k[\popo]$ and $k[\pr^3]$ are not isomorphic as graded rings),
the underlying rings are the same; in particular,
$k[\popo]=k[x_0,x_1,y_0,y_1]=k[\pr^3]$. A given ideal in this common underlying ring
can define non-isomorphic subschemes depending on which graded structure
we use. For example, the irrelevant ideals $(x_0,x_1)$
and $(y_0,y_1)$ corresponding to the two factors of $\pr^1$ in $\popo$
define a pair of skew lines $L_1\cong\pr^1$ and $L_2\cong\pr^1$ in $\pr^3$,
where $I(L_1)=(y_0,y_1)$ and hence $k[L_1]=k[x_0,x_1]$, and similarly
$I(L_2)=(x_0,x_1)$ and $k[L_2]=k[y_0,y_1]$.
Thus the point $P=[a_0:a_1] \times [b_0:b_1] \in \popo$
corresponds to a pair of points $P_1=[a_0:a_1] \in L_1$ and $P_2=[b_0:b_1]\in L_2$ and
the ideal $I(P)$ defines the line $L_P$ in $\pr^3$ through the points $P_1$ and $P_2$.
Given distinct points $P,Q\in\popo$, the lines $L_P$ and $L_Q$ meet if and only if
either $P_1=Q_1$ or $P_2=Q_2$; i.e., if and only if $P$ and $Q$ are both
on the same horizontal rule or both on the same vertical rule of $\popo$.

Given any single line $L\subset\pr^3$, lines $L_1\cong\pr^1$ and $L_2\cong\pr^1$ in $\pr^3$
can be found such that $I(L)$ is the ideal of a single point in $\popo$.
Likewise, for any two lines $L,L'\subset \pr^3$, lines $L_1\cong\pr^1$ and $L_2\cong\pr^1$ in $\pr^3$
can be found such that $I(L\cup L')$ is the ideal of two points in $\popo$, and
if the lines are general so are the points.
Consider three general lines $L,L'L''$. There is a unique smooth quadric $Q$ (isomorphic to $\popo$)
containing them. The lines $L,L',L''$ lie in a single ruling of $Q$,
and we can take $L_1$ and $L_2$ to be any two lines in the other ruling; with respect
to $L_1$ and $L_2$, $I(L\cup L'\cup L'')$ defines 3 general points of $\popo$.
(Note that the $\popo$ defined by $L_1$ and $L_2$ is not canonically the quadric $Q$ itself,
although $Q$ is isomorphic to $\popo$ abstractly.)
Finally, consider four general lines $L,L',L'',L'''$. Then $L,L'$ and $L''$ determine $Q$ and lie in a giving
ruling on $Q$, and $L'''$ meets $Q$ in two points. We take $L_1$ and $L_2$
to be the
lines of the other ruling through these two points. Now with respect
to $L_1$ and $L_2$, $I(L\cup L'\cup L'')$ defines four general points in $\popo$.
\end{remark}

One situation for which $I^{(m)}=I^m$ for all $m$ occurs
is the case that $I$ is a complete intersection,
meaning that $I$ has a set of $t$ generators, where $t$ is the codimension.
For example, suppose $I$ is the ideal of a finite set $Z$ of points of
$\pr^1 \times \cdots \times \pr^1=(\pr^1)^t=Y$. Then $\operatorname{codim}_{Y}(Z)=t$,
so $I$ is a complete intersection if it is generated by $t$ elements of $I$.
As noted in \cite[Remark 1.3]{GuMaRa} for $t=2$ (but which extends naturally to all $t \geq 2$),
an ideal $I$ of a finite set of points $Z\subset Y$ is a complete intersection
if and only if $Z$ is a rectangular array
of points (i.e., $Z=X_1\times\cdots\times X_t$ for finite sets $X_i\subset \pr^1$).

\begin{proposition}\label{completeintersections}
Let $X_1,\ldots,X_t \subseteq \pr^1$ be finite sets of points, and
let $I$ be the ideal of
$Z = X_1 \times X_2 \times \cdots \times X_t \subseteq
\pr^1 \times \cdots \times \pr^1$. Then $I^m = I^{(m)}$ for all $m \geq 1$.
\end{proposition}

\begin{proof} Under these hypotheses, $I = I(X_1)R + \cdots + I(X_t)R$
with $R = k[\pr^1 \times \cdots \times \pr^1]$ and $I(X_i)$ is the
defining ideal of $X_i$ in $k[\pr^1]$.  The ideal $I$ is
then a complete intersection.  For any complete intersection $I$, we
have $I^m = I^{(m)}$ for all $m \geq 1$
(see \cite[Lemma 5, Appendix 6]{ZS}).
\end{proof}

\subsection{Hilbert functions and points in multiplicity 1 generic position}
Let $Z\subseteq\pr^N$ be the subscheme defined by a homogeneous ideal $I$ in $k[\pr^N]$.
We recall that the {\it Hilbert function} $H_Z$ of $Z$ is defined to be
$H_Z(t)=\dim k[\pr^N]_t-\dim I_t$, where for a graded module $M$, $M_t$
denotes the homogeneous piece of degree $t$.
Similarly, recall that the Hilbert function $H_Z$ of a subscheme $Z\subseteq \popo$
is defined to be $H_Z(i,j)=\dim k[\popo]_{(i,j)}-\dim I(Z)_{(i,j)}$.

Consider a finite set of points $Z \subseteq \popo$ (regarded as a reduced
subscheme). We will say $Z$ has {\em generic} Hilbert function if
\[H_Z(i,j) = \min\{\dim R_{(i,j)}, |Z|\} = \min\{(i+1)(j+1), |Z|\}.\]
It is well known that points with generic Hilbert function are general;
i.e, for each $s\geq1$, there is a non-empty open subset of $U_s\subset (\popo)^s$
consisting of distinct ordered sets of $s$ points of $\popo$
with generic Hilbert function (see, for example, \cite{VTgen}).
In particular, subschemes $Z=P_1+\cdots+P_s$
consisting of $s$ distinct points for which
every subset of the points has generic Hilbert function
are general.

We will say that a set of $s$ distinct points $P_1,\ldots,P_s$ are
{\em multiplicity 1 generic} or are in {\em multiplicity 1 generic position} if for every
subscheme $Z=m_1P_1+\cdots+m_sP_s$ with $0\leq m_i\leq 1$, $Z$ has
generic Hilbert function. Thus being multiplicity 1 generic holds for general points.
Note that points $P_1,\ldots,P_s\in\popo$ being generic is not the
same as being multiplicity 1 generic. To explain, let
$\mathbb{K} \subseteq k$ be a subfield. Then there is a natural
inclusion $\pr^1_{\mathbb{K}}\subseteq \pr^1_k$, and we say that
$P_1,\ldots,P_s\in\pr^1_k\times\pr^1_k=(\popo)_k$ are {\em
generic} if $P_i\in (\popo)_{k_i}\setminus
(\popo)_{k_{i-1}}$ for each $i$, where
$k_0\subsetneq k_1\subsetneq \cdots\subsetneq k_s= k$
is a tower of algebraically closed fields
such that $k_0$ is the algebraic closure $\overline{k'}$ of the prime field
$k'$ of $k$. Thus for example, if
$C\subset\pr^2$ is an irreducible reduced cubic with a double
point, and if we pick points $p_1,\ldots,  p_8\in C$ such that no
three are collinear and no six lie on a conic but such that $p_1$
is the double point, then the points are multiplicity 1 generic but not
generic. On the other hand, $s$ generic points are multiplicity 1 generic.

\begin{example}
Any single point of $\popo$ is in multiplicity 1 generic position.
Two points of $\popo$ are in multiplicity 1 generic position if and only if
they are not both on the same horizontal or vertical rule of
$\popo$. As a consequence, if $s\geq 3$ points are in
multiplicity 1 generic position, then no two of them lie on the same
horizontal or vertical rule. For $s=3$, the converse is also true
(since any such three points are equivalent under an isomorphism of
$\popo$), but for $s\geq4$ points the condition that
no two lie on the same horizontal or vertical rule
is not sufficient to ensure that the points are in multiplicity 1 generic position.
(This is because given three points in multiplicity 1 generic position, there is,
up to multiplication by scalars, a unique form of degree
$(1,1)$ which vanishes on the three points. In order for four points
to be in multiplicity 1 generic position, the fourth point cannot be in the zero-locus
of the $(1,1)$-form associated to the other three points.)
\end{example}

\subsection{Divisors on blow ups and a connection to $\pr^2$}
Given a finite set of distinct points $P_1,\ldots,P_s\in\popo$,
let $\pi:X\to \popo$ be the birational morphism obtained by
blowing up the points $P_i$. Let $\operatorname{Cl}(X)$ be the divisor
class group of $X$. Let $H$ and $V$ be the pullback to $X$ of
general members of the rulings on $\popo$ (horizontal and
vertical, respectively), and for each point $P_i$ let $E_i$ be the
exceptional divisor of the blow up of $P_i$. Every divisor is
linearly equivalent to a unique divisor of the form
$aH+bV-m_1E_1-\cdots -m_sE_s$. Because of this, we can regard ${\rm
Cl}(X)$ as the free abelian group on the set
$\{H,V,E_1,\ldots,E_s\}$. This basis is called an {\em exceptional
configuration}. In particular, when we have a divisor of the form
$aH+bV-m_1E_1-\cdots- m_sE_s$, we will leave it to context whether
we really mean a divisor or its linear equivalence class in ${\rm
Cl}(X)$. We also recall that the intersection form on ${\rm
Cl}(X)$ is determined by $H\cdot E_i=V\cdot E_i=H^2=V^2=
E_i \cdot E_j = 0$ for all $i\neq j$, and $-H\cdot V = E^2_i = -1$ for $i>0$.

Given a divisor $F$ on $X$, it will be convenient to write $h^i(X,
F)$ in place of $h^i(X, \mathcal O_X(F))$, and we will refer to a
divisor class as being {\em effective\/} if it is the class of an
effective divisor. We also sometimes say by ellipsis that a divisor
is effective when we mean only that it is linearly equivalent to an effective divisor.
(If we were ever to mean that a divisor is actually effective and not just linearly equivalent to
an effective divisor, we would say the divisor is strictly effective.)
We denote the subsemigroup of classes of
effective divisors by $\EFF(X)\subseteq \operatorname{Cl}(X)$.  We recall
that a divisor or divisor class $D$ is {\em nef} if $D\cdot C\geq 0$
for every effective divisor $C$, and we denote the subsemigroup of classes of
nef divisors by $\NEF(X)\subseteq \operatorname{Cl}(X)$.

Problems involving fat points $Z=\sum_im_iP_i$ with support at
distinct points $P_i\in\popo$ can be translated into problems
involving divisors on $X$. Given $I=I(Z)$ and $(i,j)$, then as a
vector space $I(Z)_{(i,j)}$ can be identified with
$H^0(X,iH+jV-\sum_im_iE_i)$, which itself can be regarded as a
vector subspace of the space of sections $H^0(\popo,\mathcal O
_{\popo}(i,j))$. Thus given $(i,j)$, it is convenient to define the
divisor $F(Z,(i,j))=iH+jV-\sum_im_iE_i$, in which case we have,
under the identifications above,
\[I(Z)=\bigoplus_{i,j}I(Z)_{(i,j)}
=\bigoplus_{i,j}H^0(X, F(Z,(i,j))).\]

\begin{remark}\label{relation}
It can be useful to reinterpret problems involving points of $\popo$
as problems involving points of $\pr^2$. Let $Y$ be a finite set of points
 $p_1,\ldots,p_s$ of $\pr^2$. Let $Z$ be the image of $Y$
under the birational transformation from $\pr^2$ to $\popo$ given by blowing
up two points $p_{s+1},p_{s+2}\in \pr^2$ such that none
of the points $p_i$, $i<s+1$ is on the line $A$ through $p_{s+1}$ and $p_{s+2}$
and blowing down the proper transform $E$ of $A$.
The divisors $L,E_1,\ldots,E_{s+2}$, where $L$ is a line and $E_i$ is the exceptional
curve obtained by blowing up the point $p_i$, give a basis of the divisor class group
$\operatorname{Cl}(X)$ for the surface $X$ obtained by blowing up the
points $p_i$, also called an exceptional
configuration. The birational transformation from $\pr^2$ to $\popo$
described above induces a birational morphism $X\to \popo$
given by contracting $E_1,\ldots,E_s, L-E_{s+1}-E_{s+2}$. We also
have an exceptional configuration on $X$ coming from blowing up points
$P_0,P_1,\ldots,P_s\in\popo$ to obtain $X$; this basis is given by
$H=L-E_{s+1},V=L-E_{s+2},E_1,\ldots,E_s,E=L-E_{s+1}-E_{s+2}$
where $H$ and $V$ give the rulings on $\popo$.
We can identify $P_i$ with $p_i$ for $i=1,\ldots,s$; $P_0$ is the point obtained by
contracting the proper transform of the line through $p_{s+1}$ and $p_{s+2}$.
Thus $H^0(X,aH+bV-m(E_1+\cdots+E_s))= H^0(X,
(a+b)L-m(E_1+\cdots+E_n)-aE_{s+1}-bE_{s+2})$.
If $I$ is the ideal of the fat points $mP_1+\cdots+mP_s$, we note that
$\alpha(I^{(m)})$ is then the least $t$ such that $t=a+b$ and $h^0(X,
(a+b)L-m(E_1+\cdots+E_s)-aE_{s+1}-bE_{s+2})>0$.

Alternatively, suppose $P_1,\ldots,P_s\in\popo$ are such that no two of the points $P_i$
lie on the same horizontal or vertical rule. Let $X\to\popo$ be the birational morphism obtained
by blowing up the points $P_i$. Then there is also a birational morphism $X\to\pr^2$.
If $H,V,E_1,\ldots,E_s$ is the exceptional configuration for $X\to\popo$, the exceptional configuration
for $X\to\pr^2$ can be taken to be $L=H+V-E_s$, $E_1'=E_1,\ldots,E_{s-1}'=E_{s-1}$,
$E_s'=H-E_s$ and $E_{s+1}'=V-E_s$.
\end{remark}

\begin{lemma}\label{exccurvesLemma}
Let $P_1,\ldots,P_s\in\popo$ be distinct points
and let $X\to\popo$ be the birational morphism obtained
by blowing these points up.
Then a divisor $C\subset X$ is a prime divisor with $C^2<0$
if and only if $C^2=C\cdot K_X=-1$ for any $s\leq 8$ generic points
and also for general sets of $s\leq 7$ points.
If $s\leq7$, then in terms of the exceptional configuration for $X\to\popo$
the classes of these curves $C$ are (up to permutations of the $E_i$ and swapping
$H$ and $V$) precisely
\begin{enumerate}
\item[] $E_1$,
\item[] $H-E_1$,
\item[] $H+V-E_1-E_2-E_3$,
\item[] $2H+V-E_1-\cdots-E_5$,
\item[] $2H+2V-2E_1-E_2-\cdots-E_6$,
\item[] $3H+V-E_1-\cdots-E_7$,
\item[] $3H+2V-2E_1-2E_2-E_3-\cdots-E_7$,
\item[] $3H+3V-2E_1-\cdots-2E_4-E_5\cdots-E_7$,
\item[] $4H+3V-2E_1-\cdots-2E_6-E_7$, and
\item[] $4H+4V-3E_1-2E_2-\cdots-2E_7$.
\end{enumerate}
\end{lemma}

\begin{proof}
Since $s\leq 8$ and the points are either general or generic,
we can regard $X\to\pr^2$ as being the blow up of
$s+1\leq9$ points $p_1,\ldots,p_{s+1}$ in $\pr^2$, and that there is
a smooth cubic curve $D\subset\pr^2$ passing through  these points.
Thus up to linear equivalence
we have $D=-K_X=3L-E_1'-\cdots-E_s'$ with respect to the exceptional configuration
$L,E_1',\ldots,E_{s+1}'$ of the morphism $X\to \pr^2$. Since $D$ is irreducible with $D^2\geq0$,
$D$ is nef, so for any prime divisor $C$ we have $D\cdot C\geq0$.
By the adjunction formula $C^2-C\cdot D=2p_C-2$ we see $C^2\geq -2$, with
$C\cdot D=1$ if $C^2=-1$ and $C\cdot D=0$ if $C^2=-2$.

There are only finitely many possible classes of reduced, irreducible curves $C$ with
$C\cdot D=0$ when $s\leq7$ (see \cite[Proposition 4.1]{GHM}). For each of these classes,
$C$ is not effective if the points $p_i$ are general, so in fact
no such $C$ is effective if $s\leq 7$ and the points $p_i$ are general.
(For example, $(L-E_1'-E_2'-E_3')\cdot D=0$; if $L-E_1'-E_2'-E_3'$ is the class
of a strictly effective divisor $C$, then the points $p_1,p_2,p_3$ are collinear
and hence not general.)
For $s=8$ there are infinitely many possible such classes
so it is not enough to assume the points are general, but if the points
are generic then there are no prime divisors $C\neq D$ with $C\cdot D=0$
(since $C\cdot D=0$ implies the coordinates of the points satisfy an algebraic relation
coming from the group law on $D$).
Thus the only prime divisors $C$ with $C^2<0$ are those that satisfy $C^2=C\cdot K_X=-1$.
Conversely, if $C$ is a divisor with $C^2=C\cdot K_X=-1$,
then by Serre duality $h^2(X, C)=h^0(X, -D-C)$ but $h^0(X,-D-C)=0$
since $D\cdot(-D-C)<0$. Now by Riemann-Roch for surfaces
we have $h^0(X, C)-h^1(X,C)=1+(C^2+D\cdot C)/2=1$
so $C$ is effective. Up to linear equivalence, if $F$ is a prime divisor
with $F\cdot D=0$, then $F=D$ (otherwise, as above, we would get an algebraic
condition on the points $p_i$) and so $D^2=0$ (hence $s=8$).
Now if $C$ is not a prime divisor, then from $D\cdot C=1$
it follows that $C=G+rD$ with $r>0$ and $D^2=0$, where $G$ is the unique component
of $C$ with $D\cdot G=1$. But then $G^2=(C-rD)^2=-1-2r<-1$, contrary to
what is proved above.

Finally, suppose $s\leq 7$. Let $C$ be a prime divisor on $X$
with $C^2=C\cdot K_X=-1$. Let $Y$ be the surface obtained
by blowing up an arbitrary point $P_{s+1}\in \popo$.
Then denoting the pullback of $C$ to $Y$ also by $C$ we have
$(C-E_{s+1})\cdot K_Y=0$ and $(C-E_{s+1})^2=-2$.
It is not hard to check that the subgroup $K_Y^\perp$ of
classes orthogonal to $K_Y$ is, for $s<7$, negative definite,
and, if $s=7$, negative semi-definite with the only classes $F$
having $F\cdot K_Y=F^2=0$ being the multiples of $K_Y$.
Thus for $s<7$ it follows by negative definiteness that there are only finitely many classes
$C$ with $(C-E_{s+1})\cdot K_Y=0$ and $(C-E_{s+1})^2=-2$
and it is not hard to find them all.
For $s=7$, the quotient $K_Y^\perp/\langle K_Y\rangle$
is negative definite so, modulo $K_Y$, there are only finitely many classes
$C$ with $(C-E_{s+1})\cdot K_Y=0$ and $(C-E_{s+1})^2=-2$.
But $C$ must satisfy $C\cdot K_Y=-1$ and $C^2=-1$,
so there is at most one such representative in each coset of
$K_Y^\perp/\langle K_Y\rangle$. Again it is not hard to find all $C$.
\end{proof}

Note that a prime divisor
$C$ with $C^2=C\cdot K_X=-1$ is called an {\it exceptional curve}.
Exceptional curves are smooth rational curves.

\begin{lemma}\label{gammaLemma}
Let $P_1,\ldots,P_s\in\popo$ be distinct points,
$I\subset k[\popo]$ the ideal generated by all
bi-homogeneous forms that vanish at all of the points $P_i$.
Let $X$ be the blow up of these $s$ points of $\popo$,
with exceptional configuration $H,V,E_1,\ldots,E_s$.
If for some $\lambda$ and $m$ we have an effective divisor
$C= \lambda(H+V)-m(E_1+\cdots+E_s)$, then
$$\gamma(I)\leq\frac{2 \lambda}{m}.$$
If moreover for some $t$ and $r$ we have a nef divisor
$D=t(H+V)-r(E_1+\cdots+E_s)$ with
$C\cdot D=0$, then
$$\gamma(I)=\frac{2 \lambda}{m}=\frac{sr}{t}.$$
\end{lemma}

\begin{proof}
If $C$ is effective, so is $lC$ and
thus $\alpha(I^{(lm)})\leq 2 \lambda l$ for all $l\geq 1$
and therefore
$$\frac{\alpha(I^{(lm)})}{lm}\leq \frac{2 \lambda l}{lm}=\frac{2 \lambda}{m}.$$

Now assume $D$ is nef.
From $C\cdot D=0$ we get
$$\frac{2 \lambda}{m}=\frac{sr}{t}.$$
Now, given $\alpha(I^{(j)})$, we can find $a\geq0$ and $b\geq0$ with
$\alpha(I^{(j)})=a+b$ such that $(I^{(j)})_{(a,b)}\neq0$.
Moreover, $C'=aH+bV-r(E_1+\cdots+E_s)$ is effective so
$C'\cdot D=t(a+b)-jrs\geq0$, hence
$$\frac{\alpha(I^{(j)})}{j}\geq \frac{rs}{t}$$
and therefore
$$\frac{rs}{t}\leq\frac{\alpha(I^{(lm)})}{lm}\leq \frac{2 \lambda l}{lm}=\frac{rs}{t}.$$
Taking the limit as $l\to\infty$ gives the conclusion.
\end{proof}

We now give the proof of Corollary \ref{boundsCor}.

\begin{proof}[{Proof of Corollary \ref{boundsCor}}]
Let $X$ be the blow up of $\popo$ at the $s$ points with exceptional configuration
$H,V,E_1,\ldots,E_s$.

The case $s=1$ follows from Proposition \ref{completeintersections} since in this case $\alpha(I)=1$,
so consider $s=2$. Then $C=D=H+V-E_1-E_2$ is effective
(since $C=(H-E_1)+(V-E_2)$ is a sum of effective divisors) and nef
(since $D=(H-E_1)+(V-E_2)$ is a sum of prime divisors, each of which
$D$ meets non-negatively). Since $C\cdot D=0$, we have $\gamma(I)=2$
by Lemma \ref{gammaLemma}.

Consider $s=3$. Then $C=H+V-E_1-E_2-E_3$ is effective (being exceptional,
by Lemma \ref{exccurvesLemma}) and
$D=3H+3V-2(E_1+E_2+E_3)=H+V+2C$ is nef with $C\cdot D=0$ so $\gamma(I)=2$.

Consider $s=4$. Then $C=4(H+V)-3(E_1+E_2+E_3+E_4)=
C_1+C_2+C_3+C_4$ is effective (being the sum of the four exceptional
curves $C_i$, where $C_i=(H+V-E_1-E_2-E_3-E_4)+E_i$) and
$D=3H+3V-2(E_1+E_2+E_3+E_4)=2C_4+(H-E_4)+(V-E_4)$ is nef with $C\cdot D=0$
so $\gamma(I)=8/3$.

Consider $s=5$. Then $C=3(H+V)-2(E_1+\cdots+E_5)=(2H+V-E_1-\cdots-E_5)+(H+2V-E_1-\cdots-E_5)$
is effective (being the sum of two exceptional curves) and
$D=10(H+V)-6(E_1+\cdots+E_5)=D_1+\cdots+D_5$ is nef (since each
$D_i=2H+2V-(E_1+\cdots+E_5)-E_i$ is a sum of two exceptionals, each of which
$D$ meets non-negatively; for example, $D_1=(H+V-E_1-E_2-E_3)+(H+V-E_1-E_4-E_5)$).
Since  $C\cdot D=0$ we have $\gamma(I)=3$.

Consider $s=6$. Then $C=12(H+V)-7(E_1+\cdots+E_6)=C_1+\cdots+C_6$
is effective (since each $C_i=2(H+V)-(E_1+\cdots+E_6)-E_i$ is exceptional)
and $D=7(H+V)-4(E_1+\cdots+E_6)=(4H+3V-2(E_1+\cdots+E_6))+(3H+4V-2(E_1+\cdots+E_6))$ is nef
(since $4H+3V-2(E_1+\cdots+E_6)=(2H+V-(E_1+\cdots+E_5))+(2(H+V)-(E_1+\cdots+E_5)-2E_6)$
is a sum of two exceptional curves, and likewise for $(3H+4V-2(E_1+\cdots+E_6))$,
each of which $D$ meets non-negatively).
Since  $C\cdot D=0$ we have $\gamma(I)=24/7$.

Consider $s=7$. Then $C=28(H+V)-15(E_1+\cdots+E_7)=C_1+\cdots+C_7$
is effective (since each $C_i=4(H+V)-2(E_1+\cdots+E_7)-E_i$ is
exceptional) and $D=15(H+V)-8(E_1+\cdots+E_7)=H+V+D_1+\cdots+D_7$
is nef (since
$4D=2C+(3H+V-(E_1+\cdots+E_7))+(H+3V-(E_1+\cdots+E_7))$ is a sum
of exceptionals, each of which $D$ meets non-negatively). Since
$C\cdot D=0$ we have $\gamma(I)=56/15$.

Consider $s=8$. In this case
$C=D=2(H+V)-(E_1+\cdots+E_8)=-K_X$ is effective, since 8 points
impose at most 8 conditions on the 9 dimensional space of forms of
degree $(2,2)$. Since the blow up $X$ of $\popo$ at 8 general points
is a blow up of $\pr^2$ at 9 general points, and since there is an irreducible cubic
through 9 general points of $\pr^2$, we see that $-K_X$ is nef. Since $C\cdot
D=0$, we have $\gamma(I)=4$.

Now assume $s\geq 9$.
Let $C=d(H+V)-m(E_1+\cdots+E_s)$. If $C^2>0$, then $tC$ is effective for $t\gg0$,
so by Lemma \ref{gammaLemma} we have $\gamma(I)\leq 2d/m$.
It follows that $\gamma(I)\leq \sqrt{2s}$.
It is easy to compute $\alpha(I)$ for any given $s$.
In fact, since the points are general, they impose independent conditions
on forms of every bi-degree $(i,j)$; i.e., there are forms
of bi-degree $(i,j)$ vanishing at the $s$ points if and only if $(i+1)(j+1)> s$.
But for a given degree $t=i+j$, the maximum value of $(i+1)(j+1)$ occurs when
$i=j$, and so there are no forms in $I$ of total degree $t$ if
$(t/2+1)^2\leq s$. But  $(t/2+1)^2\leq s$ is equivalent to $t\leq 2(\sqrt{s}-1)$.
Thus $\alpha(I)>2(\sqrt{s}-1)$, hence we get $\sqrt{s}-1<\gamma(I)$
from the bound given in \cite[Section 2]{GHVT}.
\end{proof}


\renewcommand{\thetheorem}{\thesubsection.\arabic{theorem}}
\setcounter{theorem}{0}

\section{Additional results for general points of $\popo$}\label{sect3}
In this section, we consider the problem of whether $I^m=I^{(m)}$ for all $m$
when $I$ is the ideal of $s$ general points of $\popo$.
For $s=1,2,3,5$, we verify $I^m=I^{(m)}$ for all $m$.
For $s \geq 6$, we prove that $I^2\neq I^{(2)}$.
For $s=4$, computer calculations suggest that $I^2=I^{(2)}$,
but we show that $I^3\neq I^{(3)}$.

\subsection{Equality of $I^{(m)}$ and $I^m$}
We first consider the case of a set of two points $Z \subseteq \popo$
in multiplicity 1 generic position.
For this case, the problem reduces to a question of monomial ideals.

\begin{theorem}\label{2pts}
Let $I = I(Z)$ where $Z \subseteq \popo$ consists of
two points in multiplicity 1 generic position.
Then $I^{(m)}=I^m$ for all $m\geq 1$.
\end{theorem}

\begin{proof} Let $Z = P_1 + P_2$.
We can assume, after a change of coordinates, that
$I(P_1) =(x_0,y_0)$ and $I(P_2) = (x_1,y_1)$.
We then apply \cite[Lemma 4.1]{GHVT} for the conclusion.
\end{proof}

We now consider three points in multiplicity 1 generic position.

\begin{theorem}\label{3pts}
Let $I = I(Z)$ where $Z \subseteq \popo$ consists of three points in multiplicity 1 generic position.
Then $I^{(m)}=I^m$ for all $m\geq 1$.
\end{theorem}

\begin{proof}
For specificity say that the three points are $P_i=P_{i1}\times P_{i2}$, $i=1,2,3$,
for points $P_{ij}\in\pr^1$ and that
$k[\popo]=k[a,b,c,d]=k[a,b]\otimes_kk[c,d]=k[\pr^1]\otimes k[\pr^1]$.
Up to change of coordinates, we may as well assume
$P_{11}=P_{12}=[0:1]$,
$P_{21}=P_{22}=[1:1]$, and
$P_{31}=P_{32}=[1:0]$.

Since the points are multiplicity 1 generic, we know $\dim I_{(1,1)}=1$ , so there is (up to scalar multiples)
a unique form $F$ of degree $(1,1)$ in $I$. We will show that $I^{(m)}\subseteq I^{(m-1)}I+FI^{(m-1)}$ for each
$m\geq2$. Formally, we can write the right hand side as $I^{(m-1)}(I+F)$. Iterating $m-1$ times gives
$I^{(m)}\subseteq I(I+F)^{m-1}=I^m+FI^{m-1}+\cdots +F^{m-1}I$. Since $F\in I$,
we see that $F^iI^{m-i}\subseteq I^m$, hence $I^{(m)}\subseteq I^m$. But $I^m\subseteq I^{(m)}$,
so we have $I^{(m)}=I^m$.

We now show $I^{(m)}\subseteq I^{(m-1)}I+FI^{(m-1)}$. This is clear if $m=1$,
so assume $m\geq2$. We will consider $(I^{(m)})_{(i,j)}$
for various cases. If $(I^{(m)})_{(i,j)}=0$, then clearly $I^{(m)}\subseteq I^{(m-1)}I+FI^{(m-1)}$
so we may assume $(I^{(m)})_{(i,j)}\neq0$.

If $i+j<3m$, then apply B\'ezout's theorem: for any element $G\in (I^{(m)})_{(i,j)}$
the sum of the intersection multiplicities of $F$ with $G$ over all points $P\in\popo$ is at least
$3m$ since $G$ vanishes at each point $P_i$ with order at least $m$ while $F$ vanishes with
order 1, so summing over the three points gives at least $3m$. But $G$ has degree $(i,j)$
and $F$ has degree $(1,1)$, so at most $i+j$ common zeros are possible unless $F$ divides $G$.
Since $i+j<3m$, we see $F$ divides $G$, say $G=FH$. Then $H$ has degree $(i-1,j-1)$ and vanishes
at least $m-1$ times at each of the three points (since $G$ vanishes at least $m$ times and
$F$ vanishes once at each point). Thus $H\in(I^{(m-1)})_{(i-1,j-1)}$, so
$(I^{(m)})_{(i,j)}\subseteq F(I^{(m-1)})_{(i-1,j-1)}\subset I^{(m-1)}I+FI^{(m-1)}$.

Hereafter assume $i+j\geq 3m$.
If $j=0$, then $(I^{(m)})_{(i,j)}$ is the space of polynomials
in $a$ and $b$ of degree $(i,0)$ divisible by $a^mb^m(a-b)^m$.
Thus $(I^{(m)})_{(i,j)}=(I_{(3,0)})^mI_{(i-3m,0)}$, hence
$(I^{(m)})_{(i,j)}\subseteq I^m\subseteq I^{(m-1)}I$.
Similarly, if $i=0$, swapping $c$ and $d$ for $a$ and $b$ we again have
$(I^{(m)})_{(i,j)}\subseteq I^m\subseteq I^{(m-1)}I$.

Now assume $i>0$ and $j>0$, in addition to $i+j\geq 3m$.
The cases $i\geq j$ and $j\geq i$ are symmetric, so assume $i\geq j$.
We work on the surface $X$ obtained by blowing up
the points $P_i$. We have the birational morphism
$\pi:X\to\popo$ with exceptional configuration $H,V,E_1,E_2,E_3$,
with respect to which we can identify $(I^{(m)})_{(i,j)}$ with
$H^0(X, iH+jV-(m-1)E)$, where $E=E_1+E_2+E_3$.

If $1\leq j<m$, then we can write
$iH+jV-mE=(i-3m+j)H+j(2H+V-E)+(m-j)(3H-E)$.
Note that $3H-E=(H-E_1)+(H-E_2)+(H-E_3)$ is a sum of three
disjoint exceptional curves, disjoint also from
$(i-3m+j)H$ and $j(2H+V-E)$. Thus
$(i-3m+j)H+j(2H+V-E)$ is the nef part (with $|(i-3m+j)H+j(2H+V-E)|$
non-empty and fixed component free) and
$(m-j)(3H-E)$ is the negative (and fixed) part of a Zariski decomposition
of $iH+jV-mE$. The unique element of $|3H-E|$
corresponds to an element $Q\in I_{(3,0)}$, and since $m-j>0$ and $|3H-E|$
is the fixed part of $|iH+jV-mE|$, $Q$ is a factor of every element
of $(I^{(m)})_{(i,j)}$. Since $Q$ vanishes with order 1 at each point $P_1,P_2,P_3$,
we have $(I^{(m)})_{(i,j)}=Q(I^{(m-1)})_{(i-3,j)}\subset I^{(m-1)}I$, as we wanted to show.

So now we may assume that $i\geq j\geq m>1$ and $i+j\geq 3m$.
We will show that
under multiplication we have a surjection
$\mu:(I^{(m-1)})_{(i-2,j-1)}\otimes_k(I)_{(2,1)}\to (I^{(m)})_{(i,j)}$
and hence $(I^{(m)})_{(i,j)}\subset I^{(m-1)}I$.
But surjectivity of $\mu$ is equivalent to surjectivity of
the corresponding map
$\lambda:H^0(X, (i-2)H+(j-1)V-(m-1)E)\otimes H^0(X,2H+V-E)\to H^0(X,iH+jV-mE)$.

Under our assumptions, we have $(i-m)+(j-m)\geq m$ and $i-m\geq j-m\geq0$,
so we can pick integers $0\leq s\leq r\leq i-m$ and $s\leq j-m$ such that
$r+s=m$. Thus $iH+jV-mE=r(2H+V-E)+s(H+2V-E)+(i-m-r)H+(j-m-s)V$, and moreover
$r\geq 1$ (since $r\geq m/2>0$). Note also that $|2H+V-E|$ is non-empty and
fixed component free (since we can write $2H+V-E$ as a sum of three exceptional curves
$(H-E_u)+(H-E_v)+(V-E_w)$ in three different ways using
various permutations of $\{u,v,w\}=\{1,2,3\}$,
showing that none of the curves occurring as summands
is a fixed component), and likewise for $H+2V-E$.
Since $|2H+V-E|$, $|H+2V-E|$, $|H|$ and $|V|$ are non-empty and
fixed component free, $2H+V-E$, $H+2V-E$, $H$ and $V$ are nef.
Since $r\geq 1$ and $m\geq2$,
$|(i-2)H+(j-1)V-(m-1)E|=|(r-1)(2H+V-E)+s(H+2V-E)+(i-m-r)H+(j-m-s)V|$
is also non-empty and fixed component free, so $(i-2)H+(j-1)V-(m-1)E$ is nef.

As discussed in Remark \ref{relation}, we have a birational morphism
$p:X\to\pr^2$ with exceptional configuration
$L'=H+V-E_3$, $E_1'=E_1$, $E_2'=E_2$, $E_3'=H-E_3$ and $E_4'=V-E_3$,
so $H=L'-E_4'$, $V=L'-E_3'$, $E_1=E_1'$, $E_2=E_2'$ and $E_3=L'-E_3'-E_4'$.
Let $p_1,\ldots,p_4\in\pr^2$ be the points such that $E'_l=p^{-1}(p_l)$.
Because the points $P_1,P_2,P_3$ are multiplicity 1 generic, no three of the points
$p_l$ are collinear. Thus the proper transform $E_{uv}'$ of the line
through the points $p_u$ and $p_v$ for $u\neq v$ is an exceptional curve
and by contracting $E_{14}'$, $E_{24}'$, $E_{12}'$ and $E_3'$ we get another birational
morphism $X\to \pr^2$ obtained by blowing up four distinct general points $p_u''$,
this one having exceptional configuration
$L''=2L'-E_1'-E_2'-E_4'$, $E_1''=E_{14}'$, $E_2''=E_{24}'$, $E_3''=E_{34}'$,
and $E_4''=E_3'$. Note that $2H+V-E=2L'-E_1'-E_2'-E_4'=L''$.

Thus $\lambda$ can be written as
$\lambda:H^0(X, G)\otimes H^0(X,L'')\to H^0(X,L''+G)$
where $G=(i-2)H+(j-1)V-(m-1)E$ is nef.
Since $X$ is the blow up of four points $p_u''$ and therefore
$|2L''-E_1''-E_2''-E_3''-E_4''|\neq\varnothing$,
it follows by \cite[Proposition 2.4]{BH2} that
$\lambda$ is surjective, as claimed.
\end{proof}

\begin{remark}
Li and Swanson have given a criterion under which a radical ideal
$I$ in a reduced Noetherian domain has the property that $I^{(m)}
= I^m$ for all $m \geq 1$; see \cite[Theorem 3.6]{LS}. It is
possible that the criterion applies for ideals of any sets of two,
three or five multiplicity 1 generic points of $\popo$ in any characteristic,
but it seems difficult to verify. However, for a specific choice
of ground field and a specific choice of points one can use {\it
Macaulay2} to check the criterion. Irena Swanson, for example,
shared with us such a {\it Macaulay2} script, which shows over
${\mathbb Q}$ that the ideal $I$ of a reduced set of three points
in multiplicity 1 generic position in $\popo$ does satisfy the conditions of
\cite[Theorem 3.6]{LS}, whence $I^{(m)} = I^m$ for all $m \geq 1$.
\end{remark}

Let $I$ be the ideal of five multiplicity 1 generic
points $P_1,\ldots,P_5\in\popo$. We will show that $I^{(m)}=I^m$ for all $m\geq 1$.
The basic argument is the same as we used for three points
in general position, but it is now more complicated.

\begin{theorem}\label{5pts}
Let $I = I(Z)$ with $Z \subseteq \popo$ be five multiplicity 1 generic points.
Then $I^{(m)}=I^m$ for all $m\geq 1$.
\end{theorem}

\begin{proof}
We will show that $(I^{(m)})_{(i,j)}\subset I^{(m-1)}I$ for all
$i$ and $j$, and hence that $I^{(m)}\subseteq I^m$.
Since we know $I^m\subseteq I^{(m)}$, this shows equality.
By symmetry, we may assume $i\geq j$. We also know $I_{(5,0)}$
is 1-dimensional, whose single basis element is
the form $G=H_1\cdots H_5$, where $H_s$ is a form
of bi-degree $(1,0)$ defining the horizontal rule through the point $P_s$.
Any form $F\in (I^{(m)})_{(i,j)}$ restricts for each $s$
to a form of degree $j$ on $H_s$, but with order of vanishing at least $m$.
If $j<m$, then $F$ must vanish on the entire horizontal rule through each $P_s$,
and hence each $H_s$ divides $F$, so $G$ divides $F$. I.e.,
if $j<m$, then $(I^{(m)})_{(i,j)}=G(I^{(m-1)})_{(i-5,j)}\subset I^{(m-1)}I$.

We also know that $I_{(2,1)}$ is 1-dimensional, with basis a form $D$ defining
a smooth rational curve $C$ vanishing with order 1 at each point $P_s$.
Likewise, if $i+2j<5m$, then any form $F\in (I^{(m)})_{(i,j)}$ vanishes
on $C$, and hence $D$ divides $F$, so
$(I^{(m)})_{(i,j)}=D(I^{(m-1)})_{(i-2,j-1)}\subset I^{(m-1)}I$.

We now may assume that $i\geq j\geq m\geq 2$ and $i+2j\geq5m$.
This implies $2i+j\geq i+2j\geq 5m$, and it also
implies $i+j>3m$. (To see the latter, given $m\geq 2$, consider the system
of inequalities $i\geq j$, $j\geq m$, $i+j\leq 3m$. The solution set is
a triangular region in the $(i,j)$-plane with vertices $(3m/2,3m/2)$,
$(m,m)$ and $(2m,m)$. Since each vertex has $i+2j<5m$, we see
$i\geq j\geq m\geq 2$ and $i+2j\geq5m$ imply $i+j>3m$.)

There is a natural map $\mu_{(i,j)}: (I^{(m-1)})_{(i-3,j-1)}\otimes I_{(3,1)}\to (I^{(m)})_{(i,j)}$.
Since $\operatorname{Im}(\mu_{(i,j)})=(I^{(m-1)})_{(i-3,j-1)}I_{(3,1)}$,
to finish it is enough to show $(I^{(m)})_{(i,j)}\subseteq I^{(m-1)}I$
whenever $\mu_{(i,j)}$ is not surjective.
We can identify $(I^{(m)})_{(i,j)}$ with $H^0(X, A)$,
and $I_{(3,1)}$ with $H^0(X, L)$, where $A=iH+jV-mE$,
$L=3H+V-E$ and $E=E_1+\cdots+E_5$ are divisors on the blow up $X$ of
$\popo$ at the points $P_1,\ldots,P_5$ with respect to the
usual exceptional configuration $H,V,E_1,\ldots,E_5$.
Surjectivity of $\mu_{(i,j)}$ is equivalent to surjectivity
of the map $H^0(X, A-L)\otimes H^0(X, L)\to H^0(X, A)$,
which we will also denote by $\mu_{(i,j)}$.

Using Lemma \ref{exccurvesLemma}, the inequalities
$i\geq j\geq m\geq 2$, $i+2j\geq5m$, $2i+j\geq 5m$, and $i+j>3m$ show that
$A\cdot B\geq 0$ for every exceptional curve $B$ on $X$, and
hence $A$ is effective and nef
(since for a blow up $X$ of $\popo$ at five multiplicity 1 generic points,
and thus 6 general points of $\pr^2$,
using the results of \cite{GHM} one checks that
the only prime divisors of negative self-intersection are the
exceptional curves, but
any divisor meeting every exceptional curve non-negatively is
effective and nef \cite[Proposition 4.1]{GHM}).

Note that the exceptional configuration
$L, E_1'=H-E_1,E_2'=H-E_2,E_3'=H-E_3,E_4'=H-E_4,E_5'=H-E_5,E_6'=2H+V-E$
corresponds to a birational morphism
$X\to\pr^2$ obtained by blowing up 6 general points of $\pr^2$,
and that $L$ is the pullback of a line in $\pr^2$.
By \cite{GuH}, $\mu_{(i,j)}$ always has maximal rank.
Determining whether $\mu_{(i,j)}$ is surjective or injective
is now purely numerical, and by \cite[Theorem 3.4]{refF}, $\mu_{(i,j)}$
is surjective if $A-L$ is nef, unless either
$A-L=5L-2E_1'-\cdots-2E_6'=H+3V-E$ or
$A-L=t(-K_X-E_s')$ for $t>0$. Note that
$-K_X-E_s'=H+2V-E+E_s$ for $1\leq s\leq 5$
while $-K_X-E_6'=V$. Since each term $E_s$ of $A-L$ has the
same coefficient, $A-L=t(-K_X-E_s')$ is impossible for
$s\neq6$. Thus $\mu_{(i,j)}$ is surjective
if $A-L$ is nef, unless either $A-L=H+3V-E$ or $A-L=tV$ for $t>0$;
i.e., unless either $A=4H+4V-2E$ or $A=3H+tV-E$
for $t >1$. But $A=3H+tV-E$ is not relevant since we are interested
in cases with $m>1$. For the case $A=4H+4V-2E=-2K_X$,
we have surjectivity of
$H^0(X, -K_X)^{\otimes2}\to H^0(X, A)$ by
\cite[Proposition 3.1(a)]{refHa}. Thus
$(I^{(2)})_{(4,4)}=(I_{(2,2)})^2\subset I^2$.

So now it suffices to show that $(I^{(m)})_{(i,j)}\subseteq I^{(m-1)}I$
whenever $A-L$ is not nef but $A$ is nef and $m\geq2$. First we must find all such $A$.

Either by hand or using software such as Normaliz \cite{normaliz},
we can find generators for the semigroup of all
$(i,j,m)$ such that $i\geq j\geq m\geq 0$ and $i+2j\geq 5m$.
The result is that every such $(i,j,m)$ is a non-negative integer linear
combination of $(1,0,0)$, $(1,1,0)$, $(2,2,1)$, $(3,1,1)$, $(4,3,2)$, and $(5, 5, 3)$.
So consider $A=a(1,0,0)+b(1,1,0)+c(2,2,1)+d(3,1,1)+e(4,3,2)+f(5, 5, 3)$,
where here we use $(i,j,m)$ as shorthand for $iH+jV-mE$.

Note $A-L$ is nef for any $A=a(1,0,0)+b(1,1,0)+c(2,2,1)+d(3,1,1)+e(4,3,2)+f(5, 5, 3)$
with $d>0$, since $(3,1,1)=L$. So we may assume $d=0$.
However, $f(5H+5V-3E)-L=(t-1)(5H+5V-3E)+2(H+2V-E)$, where
$H+2V-E$ is an exceptional curve by Lemma \ref{exccurvesLemma}
with $(5H+5V-3E)\cdot (H+2V-E)=0$, so $A-L$ is effective but never nef
for $A=f(5H+5V-3E)$.

In contrast, $(e(4H+3V-2E)-L)\cdot (H+2V-E)<0$ for $e=1$, but
for $e>1$ we have $e(4H+3V-2E)-L=(5H+5V-3E)+(e-2)(4H+3V-2E)$ so, for $e>0$,
$A-L$ is not nef for $A=e(4H+3V-2E)$ if and only if $e=1$.
In particular, if $e>1$, then $A-L$ is nef
for $A=a(1,0,0)+b(1,1,0)+c(2,2,1)+e(4,3,2)+f(5, 5, 3)$
regardless of the values of $a$, $b$, $c$, and $f$.
However, $((4H+3V-2E)+f(5H+5V-3E)-L)\cdot (H+2V-E)<0$ for all $f\geq 0$,
$A-L$ is never nef for $A=(4H+3V-2E)+f(5H+5V-3E)$.

Similarly, for $c\geq0$, $c(2H+2V-E)-L$ is nef if and only if $c>1$, and
$(2H+2V-E)+f(5H+5V-3E)-L$ is never nef, but $(2H+2V-E)+(4H+3V-2E)-L$ is nef.
Thus the only cases with $A=c(2,2,1)+d(3,1,1)+e(4,3,2)+f(5, 5, 3)$
for which $A-L$ is not nef but $m\geq 2$ are:
$A=f(5, 5, 3)$, $f\geq 1$ ;
$A=(4,3,2)+f(5, 5, 3)$, $f\geq0$; and
$A=(2,2,1)+f(5, 5, 3)$, $f\geq1$.

The only other possible cases are obtained from these by adding
on to one of these multiples of either $(1,0,0)$ or $(1,1,0)$.
But $(A-L)+(1,0,0)$ for any of these $A$ is nef, so we do not get any
additional cases by allowing $a>0$ or $b>0$.
I.e., we must check that $(I^{(m)})_{(i,j)}\subseteq I^{(m-1)}I$
only when $(i,j,m)$ is either $(2,2,1)+f(5,5,3)$, $(4,3,2)+f(5,5,3)$ or $f(5,5,3)$.

First, we show $(I^{(m)})_{(i,j)}\subseteq I^{(m-1)}I$ holds for the cases
$f(5,5,3)$. Let $F=5H+5V-3E$. The divisor $E_6'=2H+V-E$
is linearly equivalent to the exceptional curve which is the proper transform
$C'$ of the curve above denoted as $C$. Likewise,
$H+2V-E$ is linearly equivalent to an exceptional curve;
denote this exceptional curve by $C''$.
Note that $F=2C'+(H+3V-E)=2C''+(3H+V-E)$. Thus
$(I_{(2,1)})^2I_{(1,3)}\subseteq (I^{(3)})_{(5,5)}$ and
$(I_{(1,2)})^2I_{(3,1)}\subseteq (I^{(3)})_{(5,5)}$, but
$\dim I_{(1,2)} = \dim I_{(2,1)} = 1$ and $\dim I_{(3,1)} = \dim I_{(1,3)} = 3$,
while $\dim (((I_{(2,1)})^2I_{(1,3)})\cap((I_{(1,2)})^2I_{(3,1)}))=0$ since
$F-2C'-2C''$ is not linearly equivalent to an effective divisor.
Thus $\dim (((I_{(2,1)})^2I_{(1,3)})+((I_{(1,2)})^2I_{(3,1)}))=6=\dim (I^{(3)})_{(5,5)}$,
hence $(I^{(3)})_{(5,5)}\subset I^3$.
Moreover, $F=5H+5V-3E$ is normally generated
by \cite[Proposition 3.1(a)]{refHa}, which means that
$H^0(X,F)^{\otimes n}\to H^0(X, nF)$ is surjective.
Thus $(I^{(3f)})_{(5f,5f)}=((I^{(3)})_{(5,5)})^f$ and hence
$(I^{(3f)})_{(5f,5f)}\subset (I^3)^f=I^{3f}$, as we needed to show.

Now consider $(I^{(2)})_{(4,3)}$. We have $I_{(1,2)}I_{(3,1)}\subseteq(I^{(2)})_{(4,3)}$
and $I_{(2,1)}I_{(2,2)}\subseteq(I^{(2)})_{(4,3)}$, but
$\dim I_{(1,2)}I_{(3,1)} = 3$, $\dim I_{(2,1)}I_{(2,2)} = \dim I_{(2,2)}=4$,
and $\dim ((I_{(1,2)}I_{(3,1)})\cap (I_{(2,1)}I_{(2,2)}))=\dim H^0(X, H)=2$,
so $\dim ((I_{(1,2)}I_{(3,1)})+(I_{(2,1)}I_{(2,2)}))=4+3-2 = 5=\dim (I^{(2)})_{(4,3)}$,
hence $(I^{(2)})_{(4,3)}=((I_{(1,2)}I_{(3,1)})+(I_{(2,1)}I_{(2,2)}))\subset I^2$,
as we needed to show.

Note that $4H+3V-2E=3L-E_1'-\cdots-E_5'$. Since the points are general,
$|3L-E_1'-\cdots-E_5'|$ and hence $|4H+3V-2E|$ contains the class of a smooth elliptic curve,
$Q$. Let $F=5H+5V-3E$.
Tensoring $0\to \mathcal O_X(-Q)\to \mathcal O_X\to \mathcal O_Q\to 0$ by $\mathcal O_X(Q+fF)$
and taking global sections gives
$0\to H^0(X, fF)\to H^0(X, Q+fF)\to H^0(Q, \mathcal O_Q(Q+fF))\to 0$.
Tensoring by $H^0(X, F)=\Gamma_X(F)$ and applying the natural multiplication
maps gives the following commutative diagram (see \cite{refMu}, or
\cite[Lemma 2.3.1]{GHI}):

{
\footnotesize

$$\begin{matrix}
0 & \!\to\! & H^0(X, fF)\otimes \Gamma_X(F)  & \!\to\!  & H^0(X, Q+fF)\otimes \Gamma_X(F)
& \!\to\!  &H^0(Q, \mathcal O_Q(Q+fF))\otimes \Gamma_X(F) & \!\to\!  & 0 \cr
{} & {} & \downarrow  & {} & \downarrow & {} & \downarrow & {} & {} \cr
0 & \!\to\! & H^0(X, (f+1)F)  & \!\to\!  & H^0(X, Q+(f+1)F) & \!\to\!  & H^0(Q, \mathcal O_Q(Q+(f+1)F)) & \!\to\!  & 0 \cr
\end{matrix}$$
}

Since $F-Q$ is linearly equivalent to an exceptional curve and hence $h^1(X, F-Q)=0$, the sequence
$0\to \mathcal O_X(F-Q)\to \mathcal O_X(F)\to \mathcal O_Q(F)\to 0$
is exact on global sections. Thus the map
$H^0(Q, \mathcal O_Q(Q+fF))\otimes \Gamma_X(F) \to H^0(Q, \mathcal O_Q(Q+(f+1)F))$
has the same image as $H^0(Q, \mathcal O_Q(Q+fF))\otimes \Gamma_Q(F) \to H^0(Q, \mathcal O_Q(Q+(f+1)F))$,
and the latter is surjective by \cite[Theorem 6]{refMu} (or see \cite[Proposition II.5(c)]{refHa2}).
We saw above that $F$ is normally generated, and hence that the map
$H^0(X, fF)\otimes \Gamma_X(F) \to H^0(X, (f+1)F)$ is surjective. Now apply the snake lemma to the
above diagram to conclude that
$H^0(X, Q+fF)\otimes \Gamma_X(F) \to H^0(X, Q+(f+1)F)$ is surjective.
By induction, we have surjectivity for all $f\geq0$ and hence
$(I^{(2+3f)})_{(4+5f,3+5f)}=(I^{(2)})_{(4,3)}((I^{(3)})_{(5,5)})^f\subset I^2I^{3f}=I^{2+3f}$.

Finally we consider the case of $(2,2,1)+f(5,5,3)$.  The proof here is the same as for
$(4,3,2)+f(5,5,3)$, except now $Q$ is a smooth elliptic curve linearly equivalent to
$-K_X=3L-E_1'-\cdots-E_6'$ and $F-Q$ is linearly equivalent to
the sum $C'+C''$ of two disjoint exceptional curves, so as before we have
$h^1(X,F-Q)=0$. Thus
$(I^{(1+3f)})_{(2+5f,2+5f)}=(I)_{(2,2)}((I^{(3)})_{(5,5)})^f\subset I^{1+3f}$.
\end{proof}

\subsection{Non-equality of $I^{(m)}$ and $I^m$}
While computer calculations suggest that $I^{(2)} = I^2$ for the ideal $I$
of four multiplicity 1 generic points in $\popo$, it is not hard to see that $I^{(3)}\neq I^3$.
This is because $\alpha(I)=3$, so $\alpha(I^3)=9$, but there is a unique curve
of bi-degree $(1,1)$ through any three of the four points (corresponding to
the divisors $H+V-E_1-E_2-E_3-E_4+E_i$ in Lemma \ref{exccurvesLemma}),
hence the sum of these four curves corresponds to a non-trivial form in
$(I^{(3)})_{(4,4)}$. Thus $\alpha(I^{(3)})\leq 8$, so $I^{(3)}\not\subseteq I^3$.

In fact, the case of four multiplicity 1 generic points is part of a much larger family,
namely a set $Z$ of $s$ points in multiplicity 1 generic position when $s = t^2$ for some
integer $t \geq 2$.  For this family, we can in a similar way
verify failures of containments of certain symbolic powers of the ideal $I(Z)$ of the points
in various ordinary powers of the ideal.

\begin{theorem}\label{squaregenericpoints}
Let $I = I(Z)$ where $Z \subseteq \popo$ is a set of $s = t^2$
points in multiplicity 1 generic position with $t \geq 2$.
Then for all integers $n \geq 1$,
\[ I^{((s-1)(2t-1)n)} \not\subseteq I^{2s(t-1)n+1}.\]
\end{theorem}

\begin{proof}
We begin by showing that the symbolic power $I^{((s-1)(2t-1)n)}$
has a nonzero element of bidegree $((t-1)s(2t-1)n,(t-1)s(2t-1)n)$.
For each point $P_i \in Z$, let $Y_i = Z \setminus \{P_i\}$.  Then
$Y_i$ is a set of $s-1$ points in multiplicity 1 generic position for each
$i=1,\ldots,s$ and hence
\[\dim(I(Y_i)_{(t-1,t-1)}) = \max\{t^2-|Y_i|,0\} = \max\{s-(s-1),0\}= 1.\]
Thus, for each $i =1,\ldots,s$, there is a form $F_i$ (unique up to scalar multiplication)
that vanishes at all of the points of $Y_i$. Moreover, $F_i$ does
not vanish at $P_i$.  Indeed, if $F_i(P_i) = 0$, then $F_i \in I(Z)_{(t-1,t-1)}$,
but $I(Z)_{(t-1,t-1)}=0$ since $\dim(I_{(t-1,t-1)})=\max\{t^2-|Z|,0\} = 0$.

Set $F = \prod_{i=1}^{s} F_i$.  The form $F$ has degree $((t-1)s,(t-1)s)$
and passes through all the points of $Z$ with multiplicity
at least $s-1$, so $F \in I^{(s-1)}$.  Thus
$F^{(2t-1)n} \in (I^{(s-1)})^{(2t-1)n}\subseteq I^{((s-1)(2t-1)n)}$
and $\deg F^{(2t-1)n} = ((t-1)s(2t-1)n,(t-1)s(2t-1)n)$ for each $n\geq1$.

To show $I^{((s-1)(2t-1)n)} \not\subseteq I^{2s(t-1)n+1}$, it is now enough to check that
$$(I^{2s(t-1)n+1})_{((t-1)s(2t-1)n,(t-1)s(2t-1)n)}=0.$$
Because the points of $Z$ are in multiplicity 1 generic position, then for $i+j=2(t-1)$, $i,j\geq0$,
we have $(i+1)(j+1)\leq t^2=|Z|$, so $\dim(I_{(i,j)})=0$. Thus,
viewing $I$ as a singly graded ideal, we have $\alpha(I)\geq 2t-1$, hence
$$\alpha(I^{2s(t-1)n+1})\geq (2s(t-1)n+1)(2t-1)>2s(t-1)n(2t-1)$$
and so $(I^{2s(t-1)n+1})_{(s(t-1)n(2t-1),s(t-1)n(2t-1))}=0$.
\end{proof}

We round out this section by comparing the symbolic squares and ordinary squares
of ideals of six or more points in multiplicity 1 generic position.

\begin{proposition}\label{6pts}
Let $I = I(Z)$ with $Z \subseteq \popo$ be a set of $6$
points in multiplicity 1 generic position.  Then $I^2 \neq I^{(2)}$.
\end{proposition}

\begin{proof}
Since $Z$
imposes at most $6\binom{2+1}{2}=18$ conditions on forms of
bidegree $(3,4)$, we see $\dim ((I^{(2)})_{(3,4)})\geq 2$. Thus
$\alpha(I^{(2)})\leq 7$, but using the fact that $I$ is multiplicity 1 generic
we compute that $\alpha(I)= 4$ so
$\alpha(I^2)= 8$, and hence $I^2\subsetneq I^{(2)}$.
\end{proof}

To extend this result to 7 or more points, we require \cite[Theorem 1]{VT3}.
We state only the part we need:

\begin{lemma} \label{secondsymb}
Let $Z \subseteq \popo$ be a set of $s$ points in multiplicity 1 generic
position, with defining ideal $I = I(Z)$.  If $(i,j) \not\in \{(2,s-1),(s-1,2)\}$,
then
\[\dim (I^{(2)})_{(i,j)} = \max\{0, (i+1)(j+1) - 3s\}.\]
\end{lemma}

We now proceed to the case of 7 or more points:

\begin{theorem} \label{secondpowervssecondsymbolic}
Let $I = I(Z)$ with $Z \subseteq \popo$ be a set of $s=|Z| \geq 7$
points in multiplicity 1 generic position.  Then
$I^2 \neq I^{(2)}$.
\end{theorem}

\begin{proof}
Let $I = I(Z)$.  To show that $I^2 \neq I^{(2)}$,
we find a bidegree $(i,j)$ where $(I^2)_{(i,j)} \neq (I^{(2)})_{(i,j)}$,
which we verify
by showing that the two graded pieces have different dimensions.

We divide $s$ by $2$ and by $3$ to write $s$ as
$s = 2q_1 +r_1$ and $s=3q_2+r_2$ where $0 \leq r_1 \leq 1$ and $0 \leq r_2 \leq 2$.
Because $Z$ is in multiplicity 1 generic position,
\begin{eqnarray*}
H_Z(1,q_1) &=&  \min\{2(q_1+1),2q_1+r_1\} = 2q_1+r_1 ~~\mbox{and} \\
H_Z(2,q_2) &=& \min\{3(q_2+1),3q_2+r_2\} = 3q_2+r_2.
\end{eqnarray*}
It then follows from the Hilbert function that
$\dim (I_{(1,q_1)}) = 2(q_1+1)-H_Z(1,q_1)=2-r_1$ and
$ \dim (I_{(2,q_2)}) = 3(q_2+1)-H_Z(2,q_2)=3-r_2$.
We will use this information, and Lemma \ref{secondsymb}, to compare
the ideals $I^2$ and $I^{(2)}$ in bidegree $(3,q_1+q_2)$.  We require two claims.

\noindent
{\em Claim 1.}  $\dim ((I^2)_{(3,q_1+q_2)}) \leq (2-r_1)(3-r_2)$.

\noindent
{\em Proof of Claim 1.}  We first note that
\[(I^2)_{(3,q_1+q_2)} = \!\!\!\!\!\!\!\!\!\!\!\!\sum_{
\begin{array}{c}
\scriptstyle 0 \leq a,b,c,d \\
\scriptstyle a+c=3,\ b+d=q_1+q_2
\end{array}
} \!\!\!\!\!\!\!\!\!\!\!\! I_{(a,b)}I_{(c,d)} .\]
The claim will follow if we show that
whenever $(a,b) \not\in \{(1,q_1),(2,q_2)\}$, then
$I_{(a,b)}I_{(c,d)} = 0$.
This would then show that $(I^2)_{(3,q_1+q_2)} = I_{(1,q_1)}I_{(2,q_2)}$,
and thus
\[\dim ((I^2)_{(3,q_1+q_2)})  \leq \dim (I_{(1,q_1)}) \dim (I_{(2,q_2)}) = (2-r_1)(3-r_2).\]

If $a=0$, then $I_{(a,b)} = 0$ since $Z$ is in multiplicity 1 generic position and $0 \leq b\leq
q_1+q_2 \leq s-1$. Likewise, $I_{(c,d)} = 0$ if $c=0$.

If $a =1$ and $b\neq q_1$, then there are two cases.
If $b<q_1$, then $I_{(a,b)} = 0$, since $H_Z(a,b) = \min\{(a+1)(b+1),2q_1+r_1\} =
(a+1)(b+1)$.  On the other hand, if $b>q_1$, then $I_{(c,d)}=0$ since $c=2$ and
$d=q_1+q_2-b < q_2$, so $(c+1)(d+1) \leq 3q_2 \leq 3q_2+r_2$, whence
$H_Z(c,d) = (c+1)(d+1)$. Likewise, $I_{(a,b)}I_{(c,d)} = 0$ if $c=1$ and $d\neq q_1$.

Finally, if $a\geq2$, then $c\leq 1$, so the same arguments apply.
\hfill$\Box$

\noindent
{\em Claim 2.} $\dim(I^{(2)})_{(3,q_1+q_2)} = q_2 + 4 -2r_1-r_2.$

\noindent
{\em Proof of Claim 2.}  By Lemma \ref{secondsymb}, we have
$\dim(I^{(2)})_{(3,q_1+q_2)} = \max\{0,4(q_1+q_2+1)-3s\}.$
By the definition of $q_1$ and $q_2$, we have $s \leq 2q_1+1$ and $s \leq 3q_2+2$.
So
\begin{eqnarray*}
4(q_1+q_2+1) - 3s & = & 4q_1+4q_2 + 4 - 3s \\
& = & (2q_1+1) + (2q_1+1) + (3q_2+2) + q_2 -3s \geq 0.
\end{eqnarray*}
Thus $\dim(I^{(2)})_{(3,q_1+q_2)} = 4(q_1+q_2+1)-3s$.  Now we get
\begin{eqnarray*}
4(q_1+q_2+1) - 3s & =& 4q_1+4q_2 + 4 - 2(2q_1+r_1) - (3q_2+r_2)  =q_2 + 4 -2r_1-r_2.
\end{eqnarray*}
by using the fact that $s=2q_1+r_1$ and $s = 3q_2+r_2$.
\hfill$\Box$

To complete the proof, it suffices to show that
\[\dim (I^{(2)})_{(3,q_1+q_2)} = q_2 + 4 -2r_1 -r_2 > (2-r_1)(3-r_2) \geq  \dim (I^2)_{(3,q_1+q_2)}.\]
But $q_2 + 4 -2r_1 -r_2 > (2-r_1)(3-r_2)$ is equivalent to
$q_2-1> (r_1-1)(r_2-1)$. The maximum value of $(r_1-1)(r_2-1)$ is 1, and it occurs only for
$r_1=r_2=0$, whereas $q_2-1>1$ unless $s=7$ or 8, and in both of these cases we have
$q_1-1=1\geq0\geq(r_1-1)(r_2-1)$.
\end{proof}

\begin{remark}
We cannot use the above proof for the case $s=6$ because
$q_2 + 4 -2r_1 -r_2 = (2-r_1)(3-r_2)$ when $s=6$
but the proof needs $q_2 + 4 -2r_1 -r_2 > (2-r_1)(3-r_2)$.
\end{remark}

Now, we are able to prove the main result of this paper:

\begin{proof}[Proof of Theorem \ref{thm1}]
That $I^{(m)}=I^m$ for all $m\geq1$ for $s$ general points for $s=1,2,3,5$, follows from
Theorems \ref{completeintersections}, \ref{2pts}, \ref{3pts} and \ref{5pts}, respectively.
That $I^{(m)}\neq I^m$ for some $m$ for all other $s$ follows for $s=4$
by Theorem \ref{squaregenericpoints} (apply the theorem with $t=2$),
for $s=6$ by Proposition \ref{6pts} and for $s>6$ by Theorem \ref{secondpowervssecondsymbolic}.
\end{proof}


\end{document}